\documentclass[11pt]{amsart}
\usepackage{amssymb,amsmath}
\usepackage{tikz-cd}
\usepackage{framed}
\usepackage[colorlinks=true]{hyperref}

\newtheorem{Th}{Theorem}
\newtheorem{Prop}{Proposition}
\newtheorem{Cor}{Corollary}
\newtheorem{Lem}{Lemma}
\newtheorem{Def}{Definition}

\def\bt{\begin{Th}}
\def\et{\end{Th}}
\def\bp{\begin{Prop}}
\def\ep{\end{Prop}}
\def\bc{\begin{Cor}}
\def\ec{\end{Cor}}
\def\bl{\begin{Lem}}
\def\el{\end{Lem}}
\def\bd{\begin{Def}}
\def\ed{\end{Def}}

\def\be{\begin{equation}}
\def\ee{\end{equation}}

\newcommand{\ot}{\otimes}
\newcommand{\op}{\oplus}
\newcommand{\we}{\wedge}
\newcommand{\om}{\omega}
\newcommand{\p}{\partial}

\newcommand{\R}{\mathbb{R}}
\newcommand{\C}{\mathbb{C}}
\newcommand{\g}{\mathfrak{g}}
\newcommand{\gp}{\mathfrak{p}}
\newcommand{\gs}{\mathfrak{s}}
\newcommand{\so}{\mathfrak{so}}
\newcommand{\gso}{\mathfrak{so}}
\newcommand{\gsu}{\mathfrak{su}}
\newcommand{\gsl}{\mathfrak{sl}}

\newcommand{\gu}{\mathfrak u}

\newcommand{\tr}{\operatorname{tr}}
\newcommand{\Inv}{\operatorname{Inv}}
\newcommand{\ad}{\operatorname{ad}}
\newcommand{\im}{\operatorname{im}}

\newcommand{\G}{\mathcal{G}}
\newcommand{\Aut}{\operatorname{Aut}}

\newcommand{\Hom}{\operatorname{Hom}}

\newcommand{\gr}{\operatorname{gr}}

\newtheorem{thm}{Theorem}
\newtheorem*{thm*}{Theorem}

\theoremstyle{definition}
\newtheorem{dfn}{Definition}

\theoremstyle{remark}
\newtheorem{rem}{Remark}

\begin{document}

\title{Constant Curvature Models in Sub-Riemannian Geometry}
\author{D.\ Alekseevsky, A.\  Medvedev, J.\ Slovak}
\thanks{All authors were partially supported by the project CZ.1.07/2.3.00/20.0003 of the
Operational Programme Education for Competitiveness of the Ministry
of Education, Youth and Sports of the Czech Republic; J.S. acknowledges the
support by the GACR, grant GA17-01171S}
\dedicatory{This work is dedicated to Valentin Lychagin}
\address{D.A.: Institute for Information Transmission Problems, Bolshoy
Karetny per. 19, build.1, Moscow 127051, Russia, and Masaryk University,
Department of Mathematics and Statistics, Kotlarska 2, 611 37 Brno, Czech
Republic}
\address{A.M.: SISSA, via Bonomea 265, Trieste, Italy, and Masaryk University,
Department of Mathematics and Statistics, Kotlarska 2, 611 37 Brno, Czech
Republic;}
\address{J.S.:Masaryk University,
Department of Mathematics and Statistics, Kotlarska 2, 611 37 Brno, Czech
Republic}
\begin{abstract}Each sub-Riemannian geometry with bracket generating distribution
enjoys a background structure determined by the distribution itself. At the
same time, those geometries with constant sub-Riemannian symbols determine a
unique Cartan
connection leading to their principal invariants. We provide cohomological
description of the structure of these curvature invariants in the cases
where the background structure is one of the parabolic geometries. As an
illustration, constant curvature models are discussed for certain
sub-Riemannian geometries.
\end{abstract}

\maketitle


\section{Introduction}
  
The   central   objects of  Riemannian geometry    are  Levi-Civita
    connection and associated geodesics and curvature.  There are three ways
    to define geodesic equation and study geodesics of a Riemannian
    manifold:
  
In  the  Hamiltonian  approach,   geodesic  equation    is  the  Hamiltonian
   equation in the cotangent bundle $T^*M$ with Hamiltonian $H(p) = \frac12
   g^{-1}(p,p)$ and geodesics are projections of the integral curves of the
   Hamiltonian flow to $M$.
   
In  the  Lagrangian  variational  approach, the  geodesic  equation  is
    the Euler-Lagrange equation for the length or energy functional in the
    space of curves.
   
In  the  geodesic     approach by Levi-Civita, geodesics   are  defined  as
autoparallel curves $\gamma(t)$, such that the tangent field
$\dot{\gamma}(t)$
   is parallel  with  respect  to Levi-Civita connection  (unique
torsionfree connection which preserves the metric).

Existence of  such connection follows from the fact that Riemannian metric,
considered as $SO(n)$-structure (i.e., a principal $ SO(n)$-bundle $ P \to M
= P/SO(n)$ of frames) has trivial fist prolongation.  In other words, the
bundle $P $ has a canonical $SO(n)$-equivariant trivialisation
        $\omega :  TP \to  \mathfrak{so}(n) +  \mathbb{R}^n  $ (Cartan
connection).  Projection to $ M$ of integral curves of “constant“ vector
fields $\omega^{-1} v,\, v \in \mathbb{R}^n$ are geodesics.
  
In  sub-Riemannian    geometry, which  studies   a  manifold  $M$   with  a
 sub-Riemannian metric $g$ defined on a non holonomic distribution $D$, all
 three approaches are working, but, as it was noted by A.  M.  Vershik and
 L.D.  Faddeev \cite{VF10}, they lead to different equations for
"sub-Riemannian geodesics".  As they remark, in Riemannian geometry a
geodesic may be defined as (locally) shortest curve and as a straightest
  curve,  but  in   sub-Riemannain  geometry    these  two   notions
   are not equivalent.  The definition of a sub-Riemannian geodesic as a
   shortest curve is based on the notion of Carnot-Carateodory metric
   $d(x,y)$ on a sub-Riemannian manifold $(M,D,g)$, defined as the infimum
   of length of horizontal curves, joining points $x,y \in M$.  It is used in
   optimisation control theory.  A.  Agrachev et al.,  using a power series
   decomposition of the square of the distance function, define
     the  curvature    of  sub-Riemannian metric   and use  it   
for  infinitesimal  variation  of  geodesics.

     In  non-holonomic mechanics, the  geodesics are  most important  as
      straightest curves.  The evolution of mechanical system with
      non-holonomic constrains is described by geodesic equation
      $\nabla_{\dot{\gamma}}\dot{\gamma} =0 $ for horizontal curve
      $\gamma(t)$, where $\nabla : \Gamma(D) \times \Gamma(D) \to
      \Gamma(G)$ is a partial connection, defined by the Koszul formula on a
      sub-Riemannian manifold $(M,D,g)$ with a fixed complementary to $D$
      distribution, see \cite{VF10}, \cite{DG}.  Moreover, V.V. Wagner
      proposed an extension of the sub-Riemannian metric to Riemannian one,
      by choosing some auxillary complementary distributions and use it to
      define the curvature for the partial connection such that its
      vanishing is equivalent to the flatness of the partial connection
      $\nabla$.
          
Summarizing, there are diverse approaches applying  different  types of  
connections  and curvature concepts in sub-Riemannian geometry. 
Depending on the problem one could
\begin{itemize}
\item construct a partial connection \cite{CCh09};
\item use the notion of a connection over a bundle map \cite{Lang01,Lang03};
\item extend sub-Riemannian metric to Riemannian metric \cite{FG95,DV10,Hl12};
\item use a very general variational approach to the curvature \cite{ABR}.
\end{itemize}
The goal of this paper is to provide an efficient framework to tackle
different equivalence problems in sub-Riemannian geometry.  We use
Cartan-Tanaka theory to construct a canonical Cartan connection and study
the structure of the curvature tensor.  The primary application we had in
mind during the work on the paper was the study of constant curvature
sub-Riemannian spaces.  They possess the biggest possible symmetry algebra
among the structures with the same metric symbol, see Section
\ref{S:constCurv}.  We classify constant curvature sub-Riemannian structures
in 2 particularly interesting cases: contact distributions and free 2-step
distributions.

Consider a distribution $D$ on the manifold $M$.
The sheaf $\mathcal D^{-1}=\mathcal D$
of vector fields with all values in $D$ generates the filtration by sheaves
$$
\mathcal D^j= \{[X,Y], X\in \mathcal D^{j+1}, Y\in \mathcal D^{-1}\}, \quad
j=-2,-3,\dots
.$$
We say that $D$ is a bracket generating distribution if for some $k$,
$\mathcal D^{-k}$ is the sheaf of all vector fields on $M$. In particular, there is
the corresponding filtration of subspaces $T_xM=D^{-k}_x\supset \dots \supset
D^{-1}_x$ at each $x\in M$ and the associated graded tangent space
$$
\operatorname{gr}(T_xM) = T_xM/D^{-k+1}_x\oplus \dots \oplus D^{-1}_x
$$
comes equipped with the structure of a nilpotent Lie
algebra.

A sub-Riemannian structure on $M$ is given by a metric $S$ which is defined only along $D$. We say that $(M, D, S)$ is a sub-Riemannian geometry with {\em  constant metric symbol} if $D$ is bracket generating, and the nilpotent algebra
$\operatorname{gr}T_xM$, together with the metric, is isomorphic to a fixed
graded Lie algebra
$$
\g_-=\g_{-k}\oplus \dots \oplus\g_{-1}
$$
with a fixed metric $\sigma$ on $\g_{-1}$.

In the sequel we shall deal with sub-Riemannian geometries with constant
metric symbols only. Under this assumption, we can employ the Cartan-Tanaka theory. We define
$$
\g_0\subset \mathfrak{so}(\g_{-1})
$$
to be the Lie algebra of the Lie group $G_0$ of all
automorphisms of the graded nilpotent algebra $\g_-$ preserving the metric
$\sigma$ on $\g_{-1}$, i.e. the algebra of certain derivations on $\g_-$.
The action of the derivations extends the Lie algebra structure on $\g_-$ to
the Lie algebra
\begin{equation}\label{E:g}
\g= \g_{-k}\oplus \dots \oplus \g_{-1}\oplus \g_0
.
\end{equation}

The graded bundle $\gr(TM)$ admits the canonical graded frame bundle
with structure group $\Aut_{\gr}(\g_-)$. A sub-Riemannian structure is a particular type of a \emph{filtered $G_0$-structure} that is a reduction of
the canonical graded frame bundle to the group~$G_0$ (\cite{Cap17}). In Section \ref{S:normCartConn} we show that the bundle $\G$ of orthogonal sub-Riemannian frames admits a natural normal Cartan connection. In our particular example a \emph{Cartan connection} is a form $\omega\colon T\G\to\g$ satisfying following properties:
\begin{itemize}
\item ${\omega}_p\colon T_p\mathcal{G} \to \mathfrak{g}$ is an isomorphism for all $p\in\mathcal{G}$;
\item $R_g^* (\omega) = \operatorname{Ad}_{g^{-1}}(\omega)$ where $R_g$ is a principal right action of an element $g\in G_0$;
\item  $\omega(\zeta_X)=X$ where $\zeta_X$ is the fundamental vector field corresponding to $X \in \g_0$.
\end{itemize}
The \emph{curvature form} $\Omega$ of the Cartan connection $\omega$ is a two-form defined by
\[ \Omega(\eta,\xi) = d\omega(\eta,\xi)+[\omega(\eta),\omega(\xi)],\quad \eta,\xi\in{\mathfrak
X}(\G). \]
It follows from the definition of a Cartan connection that $\Omega$ is equivariant for the principal action, meaning $R_g^* (\Omega) = \operatorname{Ad}_{g^{-1}}(\Omega)$, and horizontal, i.e.  $\Omega(\zeta_X,\xi)=0$ for arbitrary $X \in \g_0$. Therefore it is convenient to consider \emph{curvature function} $\kappa \colon \mathcal{G} \to
\Hom(\wedge^2
(\g/\g_0),\mathfrak{g})$ defined by the formula
\begin{equation*}
 \kappa(X,Y) = \Omega ( \omega^{-1}(X),
\omega^{-1}(Y)), \quad X,Y\in\g/\g_0.
\end{equation*}

It immediately follows that  $\kappa$ is equivariant function on $\G$. Under
reasonable normalization of the curvature, the
whole curvature function could be expressed through its essential part, the so
called harmonic curvature $\kappa_H\colon\G\to H^2(\g_-,\g)$, see Section
\ref{S:normCartConn} for the definition.  Here $H^2(\g_-,\g)$ is the second
Lie algebra cohomology space of $\g_-$ with values in $\g$.  Computation of
$H^2(\g_-,\g)$ is an essential step in understanding the structure of the
curvature function.

The paper is organized as follows. In Section \ref{S:normCartConn} we gather
observations which explain why any sub-Riemannian structure with a constant
metric symbol possesses a natural normal Cartan connection.  All facts listed
there are known \cite{Mor08,Yat88} and we summarize them for the convenience
of the reader.

Section \ref{S:cohom} is purely algebraic and shows how $H^2(\g_-,\g)$ could
be computed utilizing the information about cohomologies related to the
underlying distribution.  The results obtained in this section are of
general nature.  They may be applied for description of harmonic curvature
of a Cartan connection under the assumption that the Tanaka prolongation of
the non-positively graded Lie algebra $\g = \g_- \op \g_0$ is trivial.  To
give an example, according to the result by Yatsui~\cite{Yat15} and
independently by Cowling-Ottazzi~\cite{CO15}
the Tanaka prolongation of $\g$ for a sub-conformal structure
is either trivial, or $\g_-$ is the
nilpotent Iwasawa component of a real rank $1$ simple Lie algebra.

In Sections \ref{S:contact} and \ref{S:free} we restrict our attention to
contact sub-Riemannian structures and free 2-step distributions
respectively.  These two examples could be seen as a generalization of
the extensively studied 3-dimensional contact case~\cite{AB12,Bar13,FG95}.  We
compute cohomologies which reveal the structure of the harmonic curvature
function.  Then we show how to compute the harmonic curvature explicitly.

In the last Section \ref{S:constCurv} we illustrate how the algebraic
information about harmonic curvature could be used for the classification of
constant curvature spaces in sub-Riemannian geometry.  The advantage of our
method is that all computations are purely algebraic and reduce to the
basic representation theory of semisimple Lie algebras.  We provide the
classification for contact case and free 2-step case in
Theorem~\ref{T:constCurvContact} and Theorem~\ref{T:constCurvFree}
respectively.

For the convenience of the readers who are not familiar with the Cartan
connections, the appendix explains
a straightforward construction of the
normalized connections in the case of trivial prolongations of $\mathfrak
g$,
under the additional assumption of a fixed complement $D'$ to $D$ in $TM$.

\section{Normal Cartan connections associated with sub-Riemannian structures}\label{S:normCartConn}

Let us first remind the Tanaka prolongation procedure. Given a
non-positively graded Lie algebra $\g=\g_{-k}\oplus \dots\oplus
\g_{-1}\oplus \g_0$ with the entire negative
part $\g_-$ generated by $\g_{-1}$, we inductively define the vector
spaces
$$
\g_{r+1} = \{A\in \mathfrak{gl}_{r+1}(\g_{\le r});\
A([X,Y])=[A(X),Y]+[X,A(Y)],\ X,Y\in \g_-\}
,$$
i.e. we aim at the space of homogenous derivations of degree $r+1$ on
the previously defined algebra $\g_{\le r}$.
Notice the formula within the latter definition
extends the Lie bracket to a bracket on $\g_{\le r+1}$
whenever at least one of the arguments is from $\g_-$.

It turns out that if we obtain the trivial vector space
$\g_r=0$ for some $r$, then all the subsequent spaces $\g_s$, $s>r$, will be
trivial, too. We might also proceed to $r=\infty$ with all $\g_r\ne0$ and in
both cases, we call the resulting space $\hat g=\bigoplus_{r=-k}^\infty \g_r$
the Tanaka prolongation of $\g$. Finally, the brackets constructed on
the prolongation $\hat \g$ can be completed naturally to obtain the Lie algebra structure on
$\hat g$ by requesting (again inductively)
$$
[A,B](X) = [A(X),B] + [ A, B(X)]
.$$

Notice, we might also start the above inductive prolongation procedure with
$\g_-$ and $\g_0$ is then the entire algebra of graded derivations of
homogeneity zero. We call the resulting prolongation the full Tanaka
prolongation. Shrinking $\g_0$ in the full prolongation
to a smaller algebra $\tilde\g_0$ ensures often
the finiteness of the new Tanaka prolongation of $\g_-\oplus \tilde\g_0$.

Let us come back to the graded algebra $\g$ in \eqref{E:g}, i.e. $\g_0$ is
a subalgebra in $\mathfrak{so}(\g_-)$.
We are going to show that its Tanaka prolongation is trivial.

Corollary 2 of theorem 11.1 in \cite{Tan70} reveals that the
Tanaka prolongation of $\g$ must be finite (indeed Tanaka shows that the
finiteness can be reduced to the classical answer for $G$-structures when
restricting to the $\g_0$ action on $\g_{-1}$ and this is clearly of finite type
here).

Thus we may consider this prolongation $\hat \g = \g_{-k}\oplus\dots\oplus\g_{-1}\oplus
\g_0\oplus \g_1\oplus\cdots$ and
assume that $\g_1 \neq 0$.  Denote by  $\tilde\g_{-1}$ a
$\g_0$-invariant complement to  the  centralizer of  $\g_{1}$
in $\g_{-1}$. The proposition 2.5 of \cite{Yat88} implies that
the graded subalgebra $\mathfrak s\subset \hat\g$
generated by $\tilde\g_{-1}\oplus
[\tilde\g_{-1},\g_1]\oplus\g_1$ is  semisimple. By the general theory of graded
semisimple Lie algebras, there should be a grading
element $E \in \gs_0$ for $\mathfrak s$ which is
impossible   since  $\g_0$ is compact. Thus  $\g_1 =0$   and   $\hat\g=\g$.

Let us observe that the metric $\sigma$ extends canonically to the entire tensor
algebra generated by $\g_{-1}$. Since the action of $\g_0$ on the tensor
algebra is reductive, we may identify the entire $\g_-$ with an invariant
component in the tensor algebra and thus the metric $\sigma$ uniquely
extends to the entire $\g_-$ making the individual components mutually
orthogonal. This observation allows us to set
suitable normalization conditions on the curvature of a Cartan connection of
type $(\g_{-},G_0)$, where $G_0$ is a Lie group with Lie algebra $\g_0$ such
that the inclusion of $\g_0$ into the derivations of $\g_-$ integrates to a
group homomorphism $G_0\to \operatorname{Aut}_{\gr}(\g_-)$.

The Cartan connection of such a type is an affine connection on $M$
whose Cartan curvature function
can be viewed as a function on the frame bundle valued in cochains
$C^2(\g_-,\g)$ with the obvious differential
$\partial$. Now, the adjoint mapping $\partial^*$, provides the
complementary space $\operatorname{ker}\partial^*$ to the image
$\operatorname{im}\partial$. Thus the classical theory suggests the
condition $\partial^*\kappa=0$ as the right normalization for the Cartan
connection.
Additionally, the entire curvature $\kappa$ decomposes into its
homogeneous components with respect to the grading of $\g$.

This normalization satisfies the so called Condition (C) formulated
in \cite[Definition 3.10.1, p. 338]{Mor93} as a sufficient
algebraic condition for the existence of a unique Cartan connection
with a normalized curvature for geometries on filtered manifolds.
A more straightforward explanation of the construction of normalized Cartan connections has been recently published in \cite{Cap17} and  \cite{AD15}.
See also \cite{CS09} for the background on Cartan connections.

The \emph{ harmonic part} $\kappa_H$ of the
curvature is the part lying in the kernel of $\partial$. It is annihilated by both $\partial$ and $\partial^*$. Due to the algebraic Hodge theory the kernel of
the algebraic
Laplace operator $\Delta = \partial\circ\partial^* +
\partial^*\circ\partial$ corresponds to $H^2(\g_-,\g).$ As a result harmonic curvature function takes values in $H^2(\g_-,\g).$ The Bianchi identity expresses the individual homogeneity components of $\partial\kappa$ by algebraic expressions
in terms of lower homogeneities components and their derivatives. In particular, the entire curvature $\kappa$ vanishes if and only if its harmonic part is zero. We arrive at the following theorem (\cite[Theorem 1]{Mor08}
or \cite[Example 3.4 and Theorem 4.8]{Cap17}):

\begin{thm} For each sub-Riemannian manifold $(M,D,S)$ with constant metric symbol, there
is a unique Cartan connection $(\mathcal G\to M, \om)$ of type $(\g, G_0)$
with the curvature function
$\kappa:\mathcal G\to \g\otimes\Lambda^2\g_-^*$ satisfying
$\partial^*\kappa=0$. Via the Bianchi identities, the entire
curvature is obtained from its harmonic projection $\kappa_H$, i.e. the
component with $\partial \kappa_H=0$.
\end{thm}

The distribution $D$ on $M$ itself often defines a nice finite type filtered
geometry which enjoys a canonical Cartan connection, too.
Many of them belong to the
class of the parabolic geometries, for which the full Tanaka prolongation of
$\g_-$ is a semisimple Lie algebra $\bar\g$ and
$\g_-=\g_{-k}\oplus\dots\oplus\g_{-1}$ is the opposite
nilpotent radical to the parabolic subalgebra
$\mathfrak p=\bar\g_0\oplus\dots \oplus\bar\g_k\subset
\bar\g$, with $\g_0\subset \bar \g_0$.

Let us fix a graded semisimple Lie algebra $\bar\g$  and consider the graded
frame bundle $\mathcal G_0\to M$ of $\operatorname{gr}TM$.
Under some mild cohomological conditions which are listed in
Section \ref{S:cohom}, the structure group of this bundle is the
full group of graded automorphisms of $\g_-$.
Exactly as in the
sub-Riemannian case, the algebraic Hodge theory and the corresponding
normalization are available (though the theory behind is more complicated).
In particular, the codifferential $\partial^*$ is constructed by means of
the Killing form of $\bar\g$.

Again, the harmonic part of the curvature correspond to the
components isomorphic to the 
$H^2(\g_-,\bar\g)$
and they can be computed equivalently either by means of $\partial$ or
$\partial^*$.

See \cite{CS09} for
detailed background on the parabolic geometries and the following theorem on
canonical Cartan connections with structure group $P$,
where $P$ is a suitable Lie group with the Lie algebra $\mathfrak p$.

\begin{thm}  Consider a bracket generating distribution $D$ on $M$ with the constant
symbol equal to the negative part of a graded semisimple Lie algebra $\bar\g$
and the corresponding frame bundle $\mathcal G_0\to M$ of
$\operatorname{gr}TM$. Then there
is a unique Cartan connection $(\bar{\mathcal G}\to M, \om)$ of type
$(\bar\g, P)$ with the curvature function
$\bar\kappa:\bar{\mathcal G}\to \bar\g\otimes\Lambda^2\g_-^*$ satisfying
$\partial^*\bar\kappa=0$. Via the Bianchi identities, the entire
curvature is obtained from its harmonic projection $\bar\kappa_H$, i.e. the
component with $\partial \bar\kappa_H=0$.
\end{thm}

Thus, starting with a parabolic geometry equipped additionally with the
metric on the generating distribution $D$, there are the two curvatures
$\kappa_H$ and $\bar\kappa_H$ there. The goal of this papers is to find
relations between their algebraic properties.

As we shall see, the link
between these two curvatures is quite tight from this point of view.
This is due to the fact that in
the setup of the latter two theorems, the cohomologies can be computed with
respect to the standard Lie algebra cohomology differential
$\partial:C^p(\g_-,W)\to C^{p+1}(\g_-,W)$, where the $\g_-$-module $W$ is
either $\g$ or $\bar\g$. The same cohomologies can be equally
computed by the means of the adjoint codifferential $\partial^*$, but here
the codifferentials are much more different in general. Thus also the
normalizations of the curvatures $\kappa$ and $\bar\kappa$ are quite different
and we cannot expect simple explicit links between the
general values of $\kappa_H$
and $\bar\kappa_H$, without having the metric particularly well adjusted.

At the same time, clearly equivalence of sub-Riemannian structures must imply
the equivalence of the distributions.

\section{Cohomologies related to sub-Riemannian structures}\label{S:cohom}

Consider a non-positively graded Lie algebra $\g=\g_-\op\g_0$ such that the Tanaka prolongation of $\g$ is trivial.
Let $\bar{\g}$ be a graded Lie algebra such that $\bar\g_-=\g_-$ and $\g_0\subset\bar\g_0$.

One particular choice of $\bar\g$ we should keep in mind is the full Tanaka prolongation of $\g_-$ when it is finite dimensional. For a sub-Riemannian structure $(M,D,S)$ this means that the distribution $D$ defines a finite type filtered
geometry. In the case when $\bar\g$ is semisimple the geometry admits a canonical Cartan connection. This is exactly the case for 2-step free distributions which we consider in Section \ref{S:free}.

If the full Tanaka prolongation of $\g_-$ is infinite dimensional then the choice
of Lie algebra $\bar\g$ is not so obvious.  Sub-Riemannian contact
structures $(M^{2n+1},D,S)$ which we consider in Section \ref{S:contact} are
particular examples.  We shall see that a sub-Riemannian contact structure
with a constant metric symbol defines naturally a Levi-definite CR-structure
on the manifold $M^{2n+1}$.  Such CR-structures admit a canonical Cartan
connection with values in $\gsu(n+1,1)$ and we can consider
$\bar\g=\gsu(n+1,1).$

As we have seen in the previous section, it is natural to expect
some strong relation between $H^2_{>0}(\g_-,\g)$ and
$H^2_{>0}(\g_-,\bar\g)$.

\begin{thm}\label{T:3}
Let $\bar\g=\bar\g_-\op\bar\g_0\op\bar\g_+$ be a graded Lie algebra and
$\g=\g_-\op\g_0$ be a non-positively graded Lie algebra such that $\g_-=\bar
\g_-$ and $\g_0\subset\bar\g_0$.  Assume that the Tanaka prolongation of
$\g$ is trivial, i.e.  $H^1_{\ge1}(\g_-,\g)=0$.  The cohomology
$H^2_{>0}(\g_-,\g)$ as a $\g_0$-submodule is isomorphic to a direct sum of 2
parts:
	\begin{enumerate}
		\item $H^1_{>0}(\g_-,\bar\g/\g)/H^1_{>0}(\g_-,\bar\g)$,
		\\
		\item $\ker \pi_2 \subset H^2_{>0}(\g_-,\bar\g),$
	\end{enumerate}
where $\pi_2\colon H^2_{>0}(\g_-,\bar\g)\to H^2_{>0}(\g_-,\bar\g/ \g)$ is the natural projection induced by the projection in cochains  $\pi\colon C^2_{>0}(\g_-,\bar\g)\to C^2_{>0}(\g_-,\bar\g/ \g).$
\end{thm}

We refer the two components of the harmonic curvature above as the
$H^1$-part and the $H^2$-part respectively.

\begin{proof}
Let $W$ be a $\g_-$-submodule of $\g_-$-module $V$. Equivalently
we can consider the following short exact sequence:
\begin{center}
	\begin{tikzcd}
		0 \arrow{r} & W  \arrow{r} & V  \arrow{r} & V/W
		\arrow{r} & 0	.
	\end{tikzcd}
\end{center}
It induces the short exact sequence of differential complexes
\[
\begin{tikzcd}
0 \arrow{r} & C^\bullet(\g_-,W)  \arrow{r}{i} &
C^\bullet(\g_-,V)
\arrow{r}{\pi}
&C^\bullet(\g_-,V/W )  \arrow{r} & 0	
\end{tikzcd}
\]
and the long exact sequence in cohomologies
\begin{equation}\label{eq1}{\small
\begin{tikzcd}
\arrow{r} & H^n(\g_-,W)  \arrow{r}{ i} & H^n(\g_-,V)
\arrow{r}{ \pi}
& H^n(\g_-,V/W)  \ar[out=0, in=180, looseness=1.25,
overlay,swap]{dll}{\delta}
\\	
& H^{n+1}(\g_-,W)  \arrow{r}{i} & H^{n+1}(\g_-,V)
\arrow{r}{\pi}
& H^{n+1}(\g_-,V/W )  \arrow{r} & { }
\end{tikzcd}}
\end{equation}
We are going to apply the long exact sequence \eqref{eq1} to the pair
$\g\subset\bar \g$.

Let us notice that the gradings on $\g$ and $\bar\g$ induce the
gradings on the corresponding spaces of chains, and since the differential
$\partial$ respects this grading, we get grading on the cohomology spaces,
too. Moreover, we may consider the sequences
\eqref{eq1} for the individual homogeneities separately.

There is the general algebraic fact that
$\g=\g_-\oplus\g_0\oplus\dots\oplus\g_\ell$ is the Tanaka prolongation of
$\g_-\oplus\g_0$ if and only if $H^1_{>0}(\g_-,\g)$=0, see e.g.
\cite[Proposition 2.12]{Cap17}. Thus, in our case
the second and the third rows of the long exact
sequence
\eqref{eq1} are
\begin{equation}\label{eq2}{\small
\begin{tikzcd}
\arrow{r} & H^1_{>0}(\g_-,\g)=0    \arrow{r}   & H^1_{>0}(\g_-,\bar\g)
\arrow{r}{\pi_1}
& H^1_{>0}(\g_-,\bar\g / \g )  \ar[out=0, in=180, looseness=1.25,
overlay,swap]{dll}{\delta}
\\	
& H^2_{>0}(\g_-,\g)  \arrow{r}{i_2} & H^2_{>0}(\g_-,\bar\g)
\arrow{r}{\pi_2}
& H^2_{>0}(\g_-,\bar\g / \g )  \arrow{r} & { }
\end{tikzcd}}
\end{equation}
It is not hard to compute the connecting homomorphism $\delta$. Let
$\alpha\in
C^n(\g_-,\bar\g)$ such that $\pi(\alpha)$ is a representative of $h\in
H^n(\g_-,\bar\g/\g).$ This in particular means that $\partial(\alpha) \equiv 0\mod\g$. Then $\delta(h)\in H^{n+1}(\g_-,\g)$ is
represented by the
$\partial\alpha\in C^{n+1}(\g_-,\g)$.

From the exact sequence \eqref{eq2} one can see that $H^2_{>0}(\g_-,\g)$
consists of
two parts. The first part is $H^1_{>0}(\g_-,\bar\g/\g)/H^1_{>0}(\g_-,\bar\g)$, which is mapped
by
$\delta$
injectively into $H^2_{>0}(\g_-,\g)$. The second part is $\im i_2\colon
H^2_{>0}(\g_-,\g)\to H^2_{>0}(\g_-,\bar\g).$ Exactness of the sequence
implies
$\im i_2 =
\ker \pi_2$ which consists of such elements from  $ H^2_{>0}(\g_-,\bar\g)$
whose cohomology class allows a representative in $C^2_{>0}(\g_-,\g)$.
\end{proof}

Let us now explain how the claim of the latter theorem simplifies when
dealing with the case of semisimple algebras $\bar\g$, i.e., in the case of
sub-Riemannian parabolic geometries. We present two general comments first and
then we conclude this section with computation of the low homogeneities
in $H^1_{>0}(\g_-,\bar \g/\g)$.


\begin{rem}\label{rem1} Let us start with the $H^1$-part of the
curvature in the theorem \ref{T:3}.
Consider a semisimple graded Lie algebra
$\bar\g=\g_-\oplus\bar\g_0 \oplus \bar\g_1\oplus\dots \oplus\bar\g_k$. As we
noticed in the proof above, the
first cohomology $H^1_{\ge0}(\g_-,\bar\g)$ is answering the question whether
the spaces $\bar \g_r$ are those appearing in the full Tanaka prolongation of
$\bar\g_-$.

In particular, there are just two exceptions where
$H^1_{>0}(\g_-,\bar\g)\ne0$ (we use the standard notation for the classification of
parabolic subalgebras in semisimple algebras  via the choice of some of the
positive simple roots $\Delta^+$, cf. \cite{CS09}):
\begin{itemize}
\item
the projective geometries which are of type $A_l$ with gradings corresponding to $\{\alpha_1\}\subset\Delta^+$;
\item the projective
contact geometries which are of type $C_l$, $l\ge 3$,
with gradings corresponding to $\{\alpha_1\}\subset\Delta^+$.
\end{itemize}
Thus we see that for semisimple $\bar\g$ in almost all cases the
$H^1$-part of $H^2_{>0}(\g_-,\g)$ is equivalent to $H^1_{>0}(\g_-,\bar\g/\g)$.

Next, $\bar\g_0$ equals to all derivations on the graded
algebra $\g_-$, and consequently $\bar\g$ is the full prolongation of
$\g_-$, if and only if all the zero homogeneity cohomology
$H^1_{0}(\g_-,\bar\g)$ vanishes. In all such cases, the distribution $D$
itself completely determines the parabolic geometry in question and, thus,
the canonical Cartan connection, too.

If $H^1_0(\g_-,\bar\g)\ne0$, then we need
further reduction of the algebra of all derivations to $\bar\g_0$ in order
to get $\bar \g$ as the Tanaka prolongation of $\g_-\oplus \bar\g_0$.
This is the case for the following geometries only:
\begin{itemize}
\item
the length of the grading is $k=1$;
\item
in all the contact cases (we remind that every complex
simple Lie algebra admit a unique contact grading);
\item $\bar{\mathfrak g}$ is of type $A_l$ with $l\ge3$ and the
grading corresponds to $\{\alpha_1,\alpha_i\}\subset\Delta^+$, $2\le i\le l$, where $\Delta^+$ is the the set of simple roots;
\item
$\bar{\mathfrak g}$ is of type $C_l$ with $l\ge2$ and the
grading corresponds to $\{\alpha_1,\alpha_l\}\subset\Delta^+$,
\end{itemize}
see
\cite[Proposition 4.3.1, p. 426]{CS09} or \cite[Proposition 5.1, p.
473]{Yam93}). \end{rem}

\begin{rem} We come to the $H^2$-part of the curvature in the theorem
\ref{T:3}. The projection $\pi_2$ is zero whenever the cochains
representing the cohomology are valued in $\g$. In such a case, the
$H^2$-part coincides with the entire cohomology $H^2(\g_-,\bar\g)$.

Actually, the structure of $H^2(\g_-,\bar\g)$ for semisimple graded $\bar\g$ is quite well known
and positive homogeneities in the curvature are rather exceptional. A full
list of them can be found in \cite[Proposition 5.5, p. 477]{Yam93}. Moreover,
only very few of those in the list allow for curvature components valued in
$\bar\g_{\ge0}$. Except for several gradings of length $k=1$ and some of
the contact cases, there are
just six such choices for $\bar\g$:
\begin{itemize}
\item
the path geometries with $\operatorname{dim}D\ge 3$ (type $A_l,$ grading is given by $\{\alpha_1,\alpha_2\}\subset\Delta^+,$  $l\ge3$);
\item
the real forms corresponding to quaternionic contact geometries (type $C_l,$ grading is given by $\{\alpha_2\}\subset\Delta^+,$  $l\ge3$);
\item
the Borel $B_2$ case;
\item
the lowest dimensional free 2-step distribution
$\operatorname{dim}D=3$, $\operatorname{dim}M=6$ (type $B_3,$ grading is given by $\{\alpha_3\}\subset\Delta^+,$ where $\alpha_3$ is the short simple root);
\item
one more $C_l$ case corresponding to the first two simple roots $\{\alpha_1,
\alpha_2\}\subset\Delta^+$, i.e. in
the complex version this geometry includes the common correspondence space for the
projective contact and the quaternionic contact geometries.
\item
$\bar\g$ is of the type $G_2$ and the grading corresponds to $\{\alpha_1\}$,
i.e., the well known free distribution with growth vector $(2,3,5)$.
\end{itemize}
While these six types of geometries deserve further exploration,
in all other cases, the kernel of $\pi_2$ is the entire $H^2_{>0}(\g_-,\bar\g)$.
\end{rem}

We proceed with the computation of lower homogeneities in $H^1_1(\g_-,\bar\g/\g)$.  We shall denote
$\bar\g^i=\bar\g_-\oplus\bar\g_0\oplus\bar \g_1\oplus\dots\oplus\bar\g_i$, $i\ge 0$. All of them are $\g_-$-submodules in $\bar\g$. For the next two Theorems we don't require that the Tanaka prolongation of $\g$ is trivial.

\begin{thm}\label{T:cohom_deg1}
Let $\bar\g=\bar\g_-\op\bar\g_0\op\bar\g_+$ be a
graded Lie algebra. Let $\g=\g_-\op\g_0$
be a non-positively graded Lie algebra with $\g_-=\bar \g_-$ and
$\g_0\subset\bar\g_0$. Then
\be\label{eq:hom1}
H^1_1(\g_-,\bar\g/\g)=\Hom(\g_{-1},\bar\g^0/\g)/\partial(\bar\g_1+\g).
\ee

If $\bar\g$ is a
semisimple graded Lie algebra with the grading element $E$ and $\g_0\subset\bar\g_0^{ss}$, where $\bar\g_0^{ss}$ is a semisimple part of $\bar\g_0$ then
\[H^1_1(\g_-,\bar\g/\g)=\Hom(\g_{-1}, \bar\g_0/(\g_0\op\R E) ).\]
\end{thm}
\begin{proof}
Applying the long exact sequence  \eqref{eq1} to the pair $\bar\g^{0}/\g\subset \bar\g/\g$ we obtain:
\[
   \begin{tikzcd}
   & \dots\arrow{r}
   & \bar\g^1/\bar\g^0=H^0_{1}(\g_-,\bar\g/\bar\g^{0} )
   \ar[out=0, in=180, looseness=1.25,
   overlay,swap]{dll}{\delta}
   \\
   H^1_{1}(\g_-,\bar\g^0/\g) \arrow{r} &
   H^1_{1}(\g_-,\bar\g/\g) \arrow{r} &
   H^1_{1}(\g_-,\bar\g/\bar\g^0) =0.
   \end{tikzcd}
\]
In the formula above $H^1_{1}(\g_-,\bar\g^0/\g)=\Hom(\g_{-1},\bar\g^0/\g)$ and $\delta:\bar\g^1/\bar\g^0 \to H^1_{1}(\g_-,\bar\g^0/\g)$ is induced by $\partial$ in $C(\g_-,\bar\g/\g)$. Since $\partial(\bar\g^0)=0$ and $\partial(\bar\g_1)\subset \Hom(\g_{-1},\bar\g^0/\g)$ in $C(\g_-,\bar\g/\g)$ we get that  \[\delta(\bar\g^1/\bar\g^0)=\partial(\bar\g_1+\bar\g^0)=\partial(\bar\g_1+\g)\]  and the formula \eqref{eq:hom1} holds.

Assume now that  $\bar\g$ is semisimple with the grading element $E$ and $\g_0\subset\bar\g_0^{ss}$. For any $A\in \g_0^{ss}$ we have \[\tr\left((\ad E \ad A)|_{\bar\g_i}\right)=i\cdot\tr\left(\ad A|_{\bar\g_i}\right)=0.\]
If we define by $W$ the orthogonal complement to $E$ in $\bar\g_0 $ with respect to the Killing form then it follows that $\g_0^{ss}\subset W$. Since $\g\subset\g_0^{ss}\subset W$ the space $\Hom(\g_{-1},\bar\g^0/\g)$ can be decomposed as
\[
\Hom(\g_{-1},\R E+\g) + \Hom(\g_{-1},W/\g)
\]
Let $B$ be the Killing form of $\bar\g$. For every $v\in \bar\g_{1}$, $w\in \bar\g_{-1}$ we have
\[ B([v,w],E)=B(v,[w,E])=B(v,w), \]
since $[w,E]=-[E,w]=w$. Therefore $\partial(v)(w)=B(v,w)E \mod{W}.$

We see that the projection of $\partial(\bar \g_1+\g)$ onto
$\Hom(\g_{-1},\R E+\g)$ is bijective. Therefore
\[
H^1_1(\g_-,\bar\g/\g)
=
\Hom(\g_{-1},\bar\g^0/\g)/\Hom(\g_{-1},\R E+\g)
=
\Hom(\g_{-1}, \bar\g_0/(\g_0\op\R E) )
\]
\end{proof}

\begin{thm}\label{T:cohom_deg2}
Let $\bar\g=\bar\g_-\op\bar\g_0\op\bar\g_+$ be a
graded Lie algebra. Let $\g=\g_-\op\g_0$
be a non-positively graded Lie algebra with $\g_-=\bar \g_-$ and
$\g_0\subset\bar\g_0$. Then
\begin{enumerate}
\item\label{en1}
For $ i > 0 $ we have
\[
H^1_{i+1}(\g_-,\bar\g/\g)
=
\ker\delta:H^1_{i+1}(\g_-,\bar\g/\bar\g^{i-1} )
\to H^2_{i+1}(\g_-,\bar\g^{i-1}/\g)
\]
where $\delta$ is induced by $\partial$ in $C^\bullet_{i+1}(\g_-,\bar\g^{i-1}/\g)$
and
\[
H^1_{i+1}(\g_-,\bar\g/\bar\g^{i-1})
=
\Hom(\g_{-1},\bar\g^i/\bar\g^{i-1})/
\partial(\bar\g_{i+1}+\bar\g^{i-1}).
\]
\item\label{en3}
If $k$ is the highest homogeneity in $\bar\g$ then $H^1_{> k+1}(\g_-,\bar \g/\g)=0$.
\item\label{en2}
If $\bar\g$ is a
semisimple graded Lie algebra with the grading element $E$ and $\g_0\subset\bar\g_0^{ss}$, where $\bar\g_0^{ss}$ is a semisimple part of $\bar\g_0$ then in homogeneity 2 we have
\[
H^1_2(\g_-,\bar\g / \g)\subseteq\operatorname{Sym}(\g_{-1},\bar\g_1)
\]
where $\operatorname{Sym}(\g_{-1},\bar\g_1)\subset \Hom(\g_{-1},\bar\g_1)$
denotes the subspace of
symmetric tensors with respect to the Killing form of $\bar\g$.

Moreover if $\bar\g_0=\g_0\op\R E$ or $\partial(\g_{-2}^*)=\we^2 \g_{-1}^* $ then
\[ H^1_2(\g_-,\bar\g / \g)=\operatorname{Sym}(\g_{-1},\bar\g_1).\]
\end{enumerate}
\end{thm}
\begin{proof}
Applying the long exact sequence  \eqref{eq1} to  $\bar\g^{i-1}/\g\subset \bar\g/\g$ we obtain:
\be\label{es}
   \begin{tikzcd}
   H^1_{i+1}(\g_-,\bar \g^{i-1}/\g)\arrow{r} &
   H^1_{i+1}(\g_-,\bar \g/\g)  \arrow{r} &
  H^1_{i+1}(\g_-,\bar\g/\bar\g^{i-1} )
   \ar[out=0, in=180, looseness=1.25,
   overlay,swap]{dll}{\delta}
   \\
    H^2_{i+1}(\g_-,\bar\g^{i-1}/\g) \arrow{r} & \dots
    \end{tikzcd}
\ee

We show first that $H^1_{i+1}(\g_-,\bar \g^{i-1}/\g)=0$. Let $\alpha\in C^1_{i+1}(\g_-,\bar \g^{i-1}/\g)$. Such $\alpha$ has a form $\sum_{t,j}\omega^{-t}_j\ot v^j_{i+1-t}$ where $\omega^{-t}_j\in\g_{-t}^*$ and $v^j_{i+1-t}\in\bar\g_{i+1-t}$. Denote by $t_0$ the minimal $t$ such that $v^j_{i+1-t}\neq 0$ for some $j$. Due to degree $t_0\ge2$ and therefore
\[ \partial(\alpha) =\sum_j\partial(\omega^{-t_0}_j)\ot v^j_{i+1-t_0} \mod{(\bar\g^{i-t_0}/\g)} \]
can not be equal to zero since $\partial(\omega^{-t_0}_j)\neq 0$.

The computation of $H^1_{i+1}(\g_-,\bar\g/\bar\g^{i-1})$ is straightforward. Due to degree the only elements in $C^1_{i+1}(\g_-,\bar\g/\bar\g^{i-1})$ are $\Hom(\g_{-1},\bar\g^i/\bar\g^{i-1})$. All of them are obviously closed. By factoring out $\partial(\bar\g_{i+1}+\bar\g^{i-1})$ we obtain that
\[ H^1_{i+1}(\g_-,\bar\g/\bar\g^{i-1})=\Hom(\g_{-1},\bar\g^i/\bar\g^{i-1})/\partial(\bar\g_{i+1}+\bar\g^{i-1}). \]
To finish the proof of \eqref{en1} it remains to apply exact sequence \eqref{es} keeping in mind that $H^1_{i+1}(\g_-,\bar \g^{i-1}/\g)=0$. The second statement of the theorem immediately follows from the first one.

We proceed with the proof of \eqref{en2}. According to \eqref{en1} the space $H^1_{2}(\g_-,\bar\g/\g)$ is equal to the kernel of
\[
\delta:
H^1_{2}(\g_-,\bar\g/\bar\g^{0} )
\to
H^2_{2}(\g_-,\bar\g^{0}/\g)
=
H^2_{2}(\g_-)\ot (\bar\g^{0}/ \g)
\]
where
\[
H^1_{2}(\g_-,\bar\g/\bar\g^{0})
=
\Hom(\g_{-1},\bar\g^1/\bar\g^{0})/\partial(\bar\g^{2}+\bar\g^{0}).\]

Let $W$ be the orthogonal complement to $E$ in $\bar\g_0 $ with respect to the Killing form $B.$ As in the proof of Theorem \ref{T:cohom_deg1} we have the canonical splitting
\be\label{e:c2}
C^2_{2}(\g_-,\bar\g^{0}/ \g)
=
\Hom(\we^2\g_{-1}, \R E+\g)
+
\Hom(\we^2\g_{-1},W/\g ).
\ee
which induces the canonical splitting in cohomologies
\be\label{e:cc2}
H^2_{2}(\g_-,\bar\g^{0}/ \g)
=
(\we^2\g_{-1}^*/\partial\g_{-2}^*)\ot ( \R E+\g)
+
(\we^2\g_{-1}^*/\partial\g_{-2}^*)\ot (W/\g ).
\ee

Consider the map
\[
\pi\circ \partial\colon
\Hom(\g_{-1},\bar \g^{1}/\g)
\to
\Hom(\we^2\g_{-1}, \R E+\g),
\]
where $\pi$ is defined by the splitting \eqref{e:c2}. The kernel of $\pi\circ \partial$ is exactly $\operatorname{Sym}(\g_{-1},\bar\g_1).$ Moreover the image of the projection of $\partial(\bar\g_2)$ onto $\Hom(\g_{-1},\bar \g_{1})$ is mapped invectively to $\partial(\g^*_{-2})\ot E.$ Therefore the induced map
\[
\pi\circ\delta\colon
\Hom(\g_{-1},\bar\g^1/\bar\g^{0})/\partial(\bar\g^{2}+\bar\g^{0})
\to
\left(\we^2\g_{-1}^*/\partial(\bar\g_{-2}^*)  \right)\ot  ( \R E+\g).
\]
has $\operatorname{Sym}(\g_{-1},\bar\g_1)$ as its kernel. As a result $H^1_2(\g_-,\bar\g / \g)\subseteq\operatorname{Sym}(\g_{-1},\bar\g_1)$.
Finally, if $\bar\g_0=\g_0\op\R E$ or $\partial(\g_{-2}^*)=\we^2 \g_{-1}^* $ the second component in \eqref{e:cc2} vanishes. This ensures that in these particular cases
\[
H^1_2(\g_-,\bar\g / \g)= \ker \pi\circ\delta
=
\operatorname{Sym}(\g_{-1},\bar\g_1).
\]
\end{proof}

\section{Contact sub-Riemannian structures}\label{S:contact}

Let $(M, H, g)$ be a  contact sub-Riemannian manifold of
dimension $2n+1$. An arbitrary contact   form $\theta$ defines
a non-degenerate symplectic form $\omega = d \theta|_{H}$ on $H$.
Both $\theta$ and $\omega$ are defined  up to a conformal factor.

The  symbol  Lie  algebra  $\g_-$  of   the   contact   structure $H$
(that is the graded tangent  space
\[\gr(T_xM) = H_x\op(T_xM/H_x)\]
with  the induced  Lie  bracket)  is
isomorphic to the Heisenberg  Lie  algebra
\[\mathfrak{g}_- = \g_{-2}\op \g_{-1}   =  \mathbb{R}e_0 + \mathbb{R}^{2n}  \]
with  the   non-trivial  Lie  bracket $[u,v] = \omega(u,v)e_0,\,\,  u,v \in
\g_{-1}$.
Metric symbols of sub-Riemannian contact structures
(i.e. the  symbol  algebra   $\g_- = \g_{-2} + \g_{-1}$ together   with   a  metric  $g$  on  $\g_{-1}$)
is  parametrized by skew-symmetric non-degenerate
endomorphisms  $I_g = g^{-1} \circ \omega$ of $\g_{-1}$  defined  up
to  rescaling. With  respect  to  the  $g$-orthogonal Fitting decomposition
$\g_{-1} =  \g_1^{\lambda_1} \oplus \cdots \oplus \g_1^{\lambda_k}$ ,  the  endomorphism  can be  written  as
$$  I_g = \lambda_1 J_1 \oplus \cdots \oplus \lambda_k J_k $$
where   $J_j$ is a  complex  structure  in  the  (even-dimensional) vector
 space $\g_1^{\lambda_j}$ and $\lambda_1 < \lambda_2 < \cdots < \lambda_k$.
 We may rescale $I_g$ such that $\lambda_k =1$.  Such $I_g$ is canonically
 associated with the sub-Riemannian structure $(H,g)$ and define the complex
 structure $ J = J_1 \oplus \cdots \oplus J_k $ in $ \g_1 =
 \g_{-1}^{\lambda_1} \oplus \cdots \oplus \g_{-1}^{\lambda_k}$.

The algebra of  infinitesimal  automorphisms of   the   metric  symbol  algebra   $(\g_- = \g_{-2}+ \g_{-1},g)$  is
the   direct  sum  of   the unitary  algebras
$$   \mathfrak{aut}(\g_-,g)) = \mathfrak{u}(\g_{-1}^{\lambda_1})
\oplus \cdots \oplus \mathfrak{u} (\g_{-1}^{\lambda_k}) =
\mathfrak{u}(n_1) \oplus \cdots \oplus  \mathfrak{u}(n_k) $$
where $\dim \g_{-1}^{\lambda_j} = 2 n_j$.

Since  we  assume  that   the  metric  symbol  is   constant, the   sub-Riemannian  structure induces a $g$-orthogonal  decomposition
 $$  H =  V_1 \oplus \cdots \oplus V_k$$
of  the  contact  distribution    into  a  sum of   $I_g = g^{-1}\circ \omega$-invariant  distributions   s.t.    $I_g |_{V_j} = \lambda_j J_{V_j}$   and  determines  a complex  structure
$J = J_{V_1 } \oplus \cdots  \oplus J_{V_k}$  on  the  contact  distribution  $H$, i.e.  an  almost  CR  structure.

The symbol  algebra of almost  CR  contact  structure $(H,J)$ is well known:
\[\g_-  + \bar\g_0 = \g_{-2} + \g_{-1} + (\mathfrak{u}(n) + \mathbb{R}E),\]
where  $E$ is the grading element.
It is well know that  the  full prolongation    of  this   Lie  algebra   is
\[\bar{\g}=\mathfrak{su}(n+1,1)   = \g_{-2} + \g_{-1} + \bar\g_0 + \bar\g_1 + \bar\g_2.\]
In matrix  notations
\begin{equation}
\bar\g=\gsu(n+1,1)=
\label{mat} \left\{
\begin{pmatrix}
\lambda+i\mu & -y^* & i\beta
\\
x & A-\frac{2i}n \mu I_n & y
\\
i\alpha & -x^* & -\lambda+i\mu
\end{pmatrix}
\right\},
\end{equation}
where  $A\in\gsu(n)$, $x,y\in V=\R^{2n}=\C^n$, $\alpha,\beta,\lambda,\mu\in\R$. The entry $i\alpha$ has the homogeneity $-2,$ $x$ has the homogeneity
$-1,$ $\lambda$, $\mu$ and $A$ have the homogeneity zero, $y$ has the homogeneity one, and $i\beta$ has the homogeneity two.\\
 Now  we  describe  the  cohomology $H^2(\g_-,\bar \g)$ and  relate it  with  the  cohomology $H^2(\g_-, \g)$  associated  with   contact sub-Riemannian  structure   with  maximally  symmetric   metric  symbol $\g$ that is in  the  case  when $I_g =J$    and
  $\g = \g_- + \mathfrak{u}(n)$.\\

The harmonic curvature $H^2(\g_-,\bar \g)$ has two components, of homogeneity 1 and 2. If $n > 1$ then the complexification of  homogeneity one component gives two conjugate complex representations. They correspond to bilinear maps $\Lambda^2 V^{1,0}\to V^{0,1} $ and $\Lambda^2 V^{0,1}\to V^{1,0} $. This torsion is
the obstruction to the integrability of the $CR$-structure and it is proportional to the Nijenhuis tensor. The other cohomology component takes values in $\Lambda^{1,1}V\ot\gsu(n)$. We see that
\[\left(\ker \pi_2 \colon H^2_{>0}(\g_-,\bar\g)\to H^2_{>0}(\g_-,\bar\g/
		\g)\right)=H^2_{>0}(\g_-,\bar\g).\]

In the following theorem we shall associate $\g_{-1}^*$ with $\g_{-1}$ with
the help of the canonical hermitian form.  Then $\gu(n)$ is generated by $u\we_J
v= u\we v + Ju\we Jv$, where $u,v\in\g_{-1}.$

Let $e_i$ be an orthonormal basis of $\g_{-1}$ with respect to the canonical
hermitian form. Notice that $z=[e_i,Je_i]\in\g_{-2}$
does not depend on the choice of $i$.

\begin{thm}\label{thm4}
The cohomology of sub-Riemannian contact
structure with maximally symmetric symbol is a direct sum of
$H^2_{>0}(\g_-,\bar\g)$ and the component generated by the
homogeneity 2 symmetric tensors
\begin{equation} \label{t5:e1}
\alpha_{(pq)}=\sum_t\left( e_p\we_J e_t\ot e_t^* \we e_q^* +
e_q\we_J e_t\ot e_t^* \we e_p^* \right) + (Je_p \ot e_q^* + Je_q\ot e_p^*)\we z^*.
\end{equation}
\end{thm}
\begin{proof}
Apply Theorems \ref{T:cohom_deg1} and \ref{T:cohom_deg2}. The formula \eqref{t5:e1} is obtained by applying differential to the
space of tensors from
$\Hom(\g_{-1},\bar\g_1)$ symmetric with respect to the Killing form. It remains to show that $H^1_3(\g_-,\bar\g/\g)=0$.

Let $y$ be a generator of $\bar\g_2$. Consider an arbitrary element $\beta\in C^1_3(\g_-,\bar\g/\g)$:
\[ \beta = A^i z^*\ot e_i+B_j e_j^*\ot y. \]
For arbitrary $e_k,e_i\in\g_{-1}$ such that $[e_k,e_j]=0$ we have
\[ \partial \beta(e_k,e_j)= B_j [e_k,y]. \]
As a result $B_j=0$ for arbitrary $j$. Finally
\[ \partial \beta = A^i \partial(z^*)\ot e_i \mod \bar\g^0/\g \]
implies that $A^i=0$ for arbitrary $i$ as well.
\end{proof}

Let us observe, that the latter symmetric tensors $\alpha_{(pq)}$ are cochains
generating the cohomology space. The generators in the common kernel of
$\partial$ and $\partial^*$ may look more messy in general.

\section{Free 2-step sub-Riemannian structures}\label{S:free}

\subsection{Generalities}

Let $M$ be a manifold of dimension $\frac{m(m+1)}2$ with $m\ge3$. We say that distribution $D$ of dimension $m$ is a free distribution on $M$ if $D+[D,D]=TM$. The general parabolic geometry theory implies that there exists a natural regular normal Cartan connection of type $(\bar G,\bar P)$ where $\bar G$ is a connected component of $SO(m+1,m)$ and $P$ is a stabilizer of an isotropic $m$-dimensional subspace in $\R^{2m+1}$, cf \cite[Section 4.3.2]{CS09}. The Lie algebras of $\bar G$ and $\bar P$ are
\[ \bar\g=\left\{
\begin{pmatrix}
A & X & Y
\\
Z & 0 & -X^t
\\
T & -Z^t & -A^t
\end{pmatrix}
\right\}, \quad
\bar\gp=\left\{
\begin{pmatrix}
A & 0 & 0
\\
Z & 0 & 0
\\
T & -Z^t & -A^t
\end{pmatrix}
\right\},  \]
where  $A,Y,T\in \operatorname{Mat}(m,\R)$, $X,Z\in\R^n$,
$Y+Y^t=T+T^t=0$. We introduce the following basis in $\bar\g$
\[ e^{[ij]}=T_j^i-T_i^j,\,\, e^j=Z^j, \,\, e^i_j=A^i_j, \,\, e_j=X_j, \,\,
e_{[ij]}=Y_j^i-Y^j_i . \]

The metric $g$ defines a reduction of $\bar P$-principal bundle
$\bar{\mathcal{G}}$ to $G_0=SO(m,\R)$-principal bundle $ \mathcal{G} $
of
orthogonal frames. The sub-Riemannian structure on top of the
distribution $D$
could be given in terms of orthonormal frame $X_1,\dots,X_m$ on $D$.
Let's
define $X_{[ij]}=[X_i,X_j]$. Due to the fact that $D$ is a free distribution the graded symbol of $\{X_i,X_{[jk]}\}$ is given by $e_i$, $e_{[jk]}$
with the
same relations as in $\bar{\g}$. The Lie algebra $\g_0$ is $\mathfrak{so}(m,\R)$
generated by $s^i_j=e^i_j-e_i^j$. The infinitesimal model of the corresponding sub-Riemannian structure is given by
\[ \g=\g_{-2}\op \g_{-1} \op \g_{0} = \langle e_{[ij]} \rangle \op
\langle
e_{k} \rangle \op \langle s^i_j \rangle. \]

The structure of the cohomology group $H^2(\g_-,\bar\g)$  depends on $m$. For $n > 3$ the cohomology
is the highest weight component in $\g^*_{-1}\ot \g^*_{-2} \ot \g_{-2}$. This means that the complete obstruction to local flatness of a free step-$2$ distribution for $n>3$ is given by a torsion. On the contrary the obstruction to local flatness for $n=3$ is given by a curvature. The cohomology group $H^2(\g_-,\bar\g)$ in this case is the highest weight component in $\g^*_{-1}\ot \g^*_{-2} \ot \bar\g_{0}$.

The following 2 theorems give explicit description of the harmonic curvature $H^2_{>0}(\g_-,\g)$.
Again, notice that the theorems provide generators at the level of cochains
generating the cohomology space. The generators in the common kernel of
$\partial$ and $\partial^*$ would look much more messy.

\begin{thm}\label{T:freeH2}
The $H^2$-part of $H^2_{>0}(\g_-,\g)$ coincides with $H^2(\g_-,\bar \g)$ for $n>3$. If $n=3$ then $H^2$-part of $H^2_{>0}(\g_-,\g)$ is a 1-dimensional subspace of 27-dimensional $H^2(\g_-,\bar\g)$. It is generated by
\begin{equation}\label{free3_H2}
\sum_{(i,j,k)\in \mathfrak{S}_3 } (-1)^{sgn((i,j,k))} s^i_j\ot e^*_j\we e^*_{[jk]}
+ \sum_{\{i,j,k\}} e_i\ot e^*_{[ij]} \we e^*_{[ik]}
\end{equation}
\end{thm}
\begin{proof}
The first statement of the theorem is obvious since $H^2_{>0}(\g_-,\bar\g)$ takes values in $\g_-$ for $n>3$. For $n=3$ note that according to the next theorem $H_3^1(\g_-,\bar\g/\g)=0$ which implies that $\partial\colon C^1_3(\g_-,\bar\g)\to C^2_3(\g_-,\bar\g)$ never
takes values in
$C^2_3(\g_-,\g)$. Now the direct check shows that the only elements in the highest weight component of $\g^*_{-1}\ot \g^*_{-2} \ot \bar\g_{0}$ which take values in $\g_{0}$ are proportional to \eqref{free3_H2}.
\end{proof}

\begin{thm}\label{T:freeH1}
The $H^1$-part of $H^2_{>0}(\g_-,\g)$ consists of 2 subspaces:
	\begin{itemize}
		\item homogeneity $1$ subspace is generated by tensors
		\[ \quad\alpha_{(ij)}^k =
		\left( e_j \ot e^*_i + e_i \ot e^*_j +
		\sum_t \left(e_{[jt]} \ot  e^*_{[it]} + e_{[it]} \ot e^*_{[jt]}
		\right)
		\right)
		\we e^*_k
		\]
symmetric and traceless in $(i,j)$;
		\item homogeneity $2$ subspace is generated by tensors
		\[\alpha_{(pq)}=\sum_t e_t\ot (e^*_{[tp]}\we e^*_q + e^*_{[tq]}\we
		e^*_p)+\sum_{t,r} e_{[tr]}\ot e^*_{[tp]}\we e^*_{[qr]}\]
symmetric in $(p,q)$.
	\end{itemize}
\end{thm}
\begin{proof}
In homogeneity 4 we apply Theorem \ref{T:cohom_deg2} to obtain that $H^1_4(\g_-,\bar{\g}/\g)=0.$  An arbitrary element of homogeneity $3$ has the form
	\[ \beta = A^{[rt]}_p  e^p\ot e^*_{[rt]} + B^p_{[rt]} e^{[rt]}\ot
	e^*_p.\]
To compute the differential of $\beta$ we need the following commutation relations:
\[ [e_i,e^j]=e^j_i,\quad
[e_{[jk]},e^{[ki]}] =
\begin{cases}
e^i_j, & i\neq j \\
e^i_i+e^j_j, & i=j
\end{cases}\]
Using the relations above we obtain
	\begin{multline*}
	\p \beta(e_{[ij]}\we e_k)=
	A_p^{[ij]} [ e^p,e_k] + B^k_{[rt]} [ e_{[ij]}, e^{[rt]} ]
	\\
	= A_p^{[ij]} e^p_k
	+ \sum_{i<t} B^k_{[it]} [ e_{[ij]}, e^{[it]} ]
	+ \sum_{j<t} B^k_{[jt]} [ e_{[ij]}, e^{[jt]} ]
	\\
	+ \sum_{r<i} B^k_{[ri]} [ e_{[ij]}, e^{[ri]} ]
	+ \sum_{r<j} B^k_{[ri]} [ e_{[ij]}, e^{[rj]} ]
	\\
	= A_p^{[ij]} e^p_k
	+ \sum_t  B^k_{[it]} [ e^{[ti]}, e_{[ij]} ]
	+ \sum_t  B^k_{[tj]} [ e^{[tj]}, e_{[ji]} ]
	\\
	= A_p^{[ij]} e^p_k
	-  B^k_{[it]}(e^t_j+\delta^t_j e^i_i)
	-  B^k_{[tj]}(e^t_i+\delta^t_i e^j_j)
	\end{multline*}
The coefficient in front of $e^i_i$ is equal to
	\[ \delta_i^k A^{[ij]}_i - 2 B^k_{[ij]} \]
and should be equal to 0 since we want $\p \beta(e_{[ij]}\we e_k)\in \so_n(\R).$ Therefore if $k\neq i$ then $B^k_{[ij]}=0$ which implies that all $B^k_{[ij]}$ are equal to $0$. This yields
	\[ \p \beta(e_{[ij]}\we e_k)= e^r_k A_r^{[ij]} \]
and therefore $A_r^{[ij]}=0$ follows immediately from the requirement
	$\p
	\beta(e_{[ij]}\we e_k)\in
	\so_n(\R).$ We conclude that $H^1_3(\g_-,\bar{\g}/\g)=0.$

Using Theorem \ref{T:cohom_deg1} we obtain that 	$H^1_1(\g_-,\bar{\g}/\g)$ is
	a traceless in $(i,j)$ part of the space generated by elements
	\[\beta^k_{(ij)} =(e_j^i+e_i^j) \ot  e^*_k
	,\]
and that elements from $H^1_2(\g_-,\bar{\g}/\g)$ correspond to symmetric tensors from $\Hom(\g_{-1},\bar\g_1)$. The explicit formula for generators is:
	\[ \beta_{(pq)} = \frac12 \left(  e^q \ot e^*_p +  e^p\ot e^*_q -
	\sum_i\left(e^q_i \ot e^*_{[pi]}  +  e^p_i \ot e^*_{[qi]}\right)
	\right). \]
	
	In order to obtain $H^1$-part of $H^2(\g_-,\g)$ it
	remains to
	take differential of representatives in $H^1(\g_-,\bar{\g}/\g)$. We
	have
	\[ \alpha_{(ij)}^k=\partial \beta_{(ij)}^k = \left( e_j \ot e^*_i + e_i \ot e^*_j +
	\sum_t \left(e_{[jt]} \ot  e^*_{[it]} + e_{[it]} \ot e^*_{[jt]}
	\right) \right)
	\we e^*_k;
	\]
	\[\alpha_{(pq)}=\partial \beta_{(pq)}=\sum_t e_t\ot (e^*_{[tp]}\we e^*_q +
	e^*_{[tq]}\we
	e^*_p)+\sum_{t,r} e_{[tr]}\ot e^*_{[tp]}\we e^*_{[qr]}.\]
\end{proof}

\subsection{The Cartan connections}
The latter theorem provides helpful information how to proceed when
building distinguished connections via the classical exterior calculus. Let
us illustrate this by following the computations along the lines of
\cite[Section 4]{DS10}.

Suppose $n>3$. We start by fixing any local orthonormal
frame $X_1,\dots,X_n$ of $D$
and complete it to a frame of $TM$ by adding the vector fields $X_{[ij]}=
[X_i,X_j]$. Let $\{\theta^i,\theta^{[jk]}\}$ be the corresponding coframe.
Clearly, the annihilator $D^\perp\subset T^*M$
is generated by the forms $\theta^{[jk]}$,
and by the very construction,
$$
d\theta^{[jk]}=-\theta^j\wedge \theta^k \quad\operatorname{mod} D^\perp
.
$$
The canonical Cartan connection $\omega$ will consist of components $\om^i$
and $\om^{[jk]}$ (building the soldering form on $M$), and the connection
forms $\om^i_j$ (providing a matrix of forms, valued in $\mathfrak{so}(\g_{-1})$).
While the components $\om^{[jk]}=\theta^{[jk]}$ can stay
fixed, the other ones have to be adjusted according to our normalization.

Now, assume we have got the canonical connection $\om$ and deal with the data of homogeneity one.
In the chosen frame, we may write
\begin{align*}
\om^i &= \theta^i+C^i_{[jk]}\om^{[jk]}
\\
\om^i_j &= A^i_{kj}\om^k \quad\operatorname{mod} D^\perp
.\end{align*}
We want to find the right coefficients $C^i_{[jk]}$ which provide the
complement $D'$ to $D$ in $TM$, uniquely determined by the co-closed
curvature normalization. At the same time the coefficients
$A^i_{jk}$ are the Christoffel
coefficients of the restriction of the wanted metric connection to $D$.
Actually,
the latter Christoffel symbols uniquely extend to a metric connection on the
entire $TM$ and we know that there is exactly one such metric connection for
each choice of the torsion. Thus we have to focus on the conditions on
$C^i_{[jk]}$ and $A^i_{jk}$ leading to the proper normalization of the
torsion.

We know that the splitting $TM=D\oplus D'$ has to be respected by the
connection and this itself fixes the symmetric parts of the torsion
components $D\times D'\mapsto D$ and $D'\times D\mapsto D'$.  Further we
know that the $H^2$-part of the torsion is the completely tracefree
component in $D'{}^*\wedge D^*\otimes D'$.

Since our algebra $\g_0$
provides much smaller freedom in normalizations than the entire
$\bar\g_{\ge0}$, we cannot require to kill all other components of the
homogeneity one torsion.  But the description of the generators from the
Theorem \ref{T:freeH1} implies that the homogeneity one torsion
component $D'{}^*\wedge D^*\otimes D'$ can be uniquely normalized
to be symmetric (with
respect to the metric on $D'$) and including only the one trace component
aside the completely tracefree part. This pins down the choice of $D'$, i.e.
the coefficients $C^i_{[jk]}$.

The latter
requirements, including the vanishing of the
antisymmetric part, also define the derivative on $D'$ in the $D$
directions, which is a homogeneity one data, too.

To conclude the homogeneity one part of the torsion, we have to fix its
$D^*\wedge D^*\otimes D$ component. Although Theorem \ref{T:freeH1} suggests
there should be some non-trivial link with the trace part of the torsion
$D'{}^*\wedge D^*\otimes D'$, we may also choose this component to vanish
(although the curvature of the resulting Cartan connection will not be
co-closed any more finally).
Any such choice fixes the derivative on $D$
in the $D$ directions by the Koszul formula.

In the homogeneity one step of our construction, we
have used only part of the freedom in our normalization of the forms $\om^i_j$.  Now we have to write
$$
\om^i_j = A^i_{kj}\om^k + E^i_{j[kl]}\om^{[kl]}
,$$
with already known coefficients $A^i_{kj}$, but
with $E^i_{j[kl]}$ to be still determined by exploiting the homogeneity two
data.

We may proceed the same way as in homogeneity one with $D$ and $D'$ swapped.
First, the remaining freedom is used to ensure that torsion
component $D\times D^\perp\mapsto D$ will be symmetric with respect to the
metric on $D$ (remember the symmetric part is fixed by the splitting,
but we can kill the
anti-symmetric part). This component of the torsion already contributes to
the expected curvature components, cf. the
homogeneity 2 cohomology generator in
the $H^1$-part of the harmonic curvature in the above theorem.
Finally, the torsion of the restriction of the connection to $D'$
would be strictly related to the latter symmetric component, but we may also
choose this to vanish. Any choice of this torsion component again defines the
remaining part of the connection by the Koszul formula.
This complete procedure exactly uses all our freedom.

Let us observe that although we could explicitly construct the normalized
Cartan connection with the co-closed curvature, the simplified choices might
be even more useful in practice.

\section{Constant curvature spaces}\label{S:constCurv}

Constant curvature Cartan geometries are generalizations of the classical Riemannian constant sectional curvature spaces. They could be regarded as the most simple examples of Cartan geometries of a particular type.
\begin{dfn}
Let $(\mathcal G\to M, \om)$ be a Cartan connection of a sub-Riemannian structure on $M$. We say that $\om$ has a \emph{constant curvature} if the curvature function
$\kappa:\mathcal G\to \g\otimes\Lambda^2\g_-^*$ is constant.
\end{dfn}
It follows from the definition immediately that if $\om$ has a constant curvature then $\kappa$  takes values in $G_0$-invariant part of  $\g\otimes\Lambda^2\g_-^*$. In particular, a Cartan connection of a sub-Riemannian structure has constant curvature only if harmonic curvature takes values in $G_0$-invariant part of $H^2(\g_-,\g)$.
Moreover, harmonic curvature completely defines the whole curvature function. Therefore to every element in $\Inv_{G_0} H^2(\g_-,\g)$ corresponds at most one constant curvature space. In other words, there exists an injective map:
\[ \mbox{\{classes of locally equivalent constant curvature spaces\}} \to \Inv_{G_0} H^2(\g_-,\g). \]

This simple fact motivates our classification strategy for  constant curvature sub-Riemannian manifolds: first we compute $G_0$-invariant part of $H^2(\g_-,\g)$ and then provide models corresponding to elements of $H^2(\g_-,\g)$.

In this section we obtain a full classification of constant curvature spaces for maximally symmetric contact sub-Riemannian structures and free 2-step sub-Riemannian structures. We start with the maximally symmetric contact case.
\begin{thm}\label{T:constCurvContact}
Every non-flat $(2n+1)$-dimensional constant curvature space with the maximally symmetric contact symbol is locally equivalent to the sphere in $\C^{n+1}$ or pseudo-sphere in $\C^{1,n}$. The distribution in this case is a maximal complex subspace in the tangent space to the (pseudo)-sphere.
Up to scale, the metric coincides with the one which is induced from the
ambient space.
\end{thm}
\begin{proof}
As it was shown in the discussion preceding Theorem \ref{thm4} the space $H^2(\g_-,\bar\g)$ splits into two non-trivial irreducible representations of $\gu(n)$. Therefore elements which are invariant with respect to $\g_0=\gu(n)$ could exist only in $H^1$-part of $H^2(\g_-,\g)$.
According to Theorem \ref{T:cohom_deg2} \[H^1(\g_-,\bar\g/\g)=S^2(\g_{-1}^*)\]
as $\g_0$-module.
In order to compute $\gu(n)$-invariant subspace in  $S^2(\g_{-1}^*)$ we can restrict our attention to the $J$-invariant bilinear forms since $J\in \gu(n)$. This space can be identified with the space of hermitian operators on $\C^n$. Any $\gu(n)$-invariant hermitian operator is proportional to identity. This implies that $\Inv_{\g_0} H^1(\g_-,\bar\g/\g)$ is 1-dimensional.  We conclude that the family of constant curvature spaces with the maximally symmetric contact symbol is no more than 1-dimensional.
Depending on the sign of the curvature function every non-zero element in $H^1$-part of $H^2(\g_-,\g)$ can be realized either as the sphere in $\C^{n+1}$ or pseudo-sphere in $\C^{1,n}$.
\end{proof}

In terms of Lie algebras the homogeneous infinitesimal model for the sphere in $\C^{n+1}$  is a factor of
\[ \gu(n+1)= \begin{pmatrix}
iz & - u +iv \\
u+iv & U
\end{pmatrix}, \quad z\in\R,\,\, v,u\in\R^n; U\in \gu(n) \]
by the subalgebra
\[ \gu(n)= \begin{pmatrix}
0 & 0 \\
0 & U
\end{pmatrix} \]
with the distribution $D$ given by
\[ \begin{pmatrix}
0 & -u +iv \\
u+iv & *
\end{pmatrix} \]
and sub-Riemannian norm $\|u+iv\|^2 = \|u\|^2 + \|v\|^2$. The description for pseudo-sphere in $\C^{n,1}$ is analogues.

In order to attack the same problem for free 2-step distributions we need some well known results on representation theory of complex semisimple Lie algebras. We refer here to the appendix of \cite{Vinb} and use notations from there.

Consider a rank $l$ simple Lie algebra. Let $\pi_i,$ $1\le i \le l$
be the fundamental weights. For the convenience we assume that $\pi_0=\pi_{l+1}=0$. We denote by $R(\lambda)$ the
irreducible representation with the highest weight $\lambda$.  Then $R(\lambda_1)R(\lambda_2)$ defines a tensor product of representations $R(\lambda_1)$ and $R(\lambda_2)$, $S^k(\lambda)$ is a $k$-th symmetric power of $R(\lambda)$ and  $\Lambda^k(\lambda)$ is a $k$-th wedged power of $R(\lambda)$. For Lie algebras of type $B_l$, $l\ge2$ $(n=2l+1)$ we use the notation
\[
\Lambda^p(\pi_1)= \hat\pi_p =
\begin{cases}
\pi_p & 1\le p\le l-1,
\\
2\pi_l & p=l,l+1,
\\
\pi_{n-p} & l+2\le p\le 2l,
\end{cases}
\quad\hat\pi_0=\hat\pi_n=0.
\]
For Lie algebras of type $D_l$, $l\ge3$ $(n=2l)$ we use the notation
\[
\Lambda^p(\pi_1)=\hat\pi_p =
\begin{cases}
\pi_p & 1\le p\le l-2,
\\
\pi_{l-1}+\pi_l & p=l-1,l+1,
\\
\pi_{n-p} & l+2\le p\le 2l,
\end{cases}
\quad\hat\pi_0=\hat\pi_n=0,
\]
and $R(\hat\pi_l+\lambda)=R(2\pi_{l-1}+\lambda)+R(2\pi_{l}+\lambda).$
Finally we write $R_{\gsl}(\lambda)$ and  $R_{\gso}(\lambda)$ for
modules over $\gsl(n,\C)$ and $\gso(n,\C)$ respectively.

\begin{thm}\label{T:constCurvFree}
Every non-trivial sub-Riemannian constant curvature space modeled on free
2-step distribution of rank $n$, $n\ge 4$, is locally equivalent to $SO(n+1)$ or $SO(1,n)$. In this case the left-invariant distribution on the corresponding Lie algebra is
	\[ \begin{pmatrix}
	0 & \pm v \\
	v & 0
	\end{pmatrix},\quad v\in\R^n \]
with the norm $\| v \|^2 = \sum_{i=1}^n v_i^2$.

For free 2-step distributions of rank $3$ there exists a 2-dimensional family of constant curvature models.
\end{thm}
\begin{proof}
We are going to check $\g_{0}$-invariance of three different components of $H^2(\g_-,\g)$. The computations for the generic case are different from cases $n=3$, $n=4$. We consider the generic case first.

\textbf{1.} The complexification of $H^2$-part in $H^2(\g_-,\g)$ is the highest weight $\gsl(n,\C)$-module in $\Hom(\g_{-1}\we\g_{-2},\g_{-2})$. Its weight as $\gsl(n,\C)$-module is
$\pi_{n-1}+\pi_{n-2}+\pi_2$. We need to compute its decomposition as
$\gso(n,\R)$ module. First, we have
\begin{align*}
	R_{\gsl}(\pi_{n-1})R_{\gsl}(\pi_{n-2}) =&
	R_{\gsl}(\pi_{n-1}+\pi_{n-2})
	+R_{\gsl}(\pi_{n-3}),
	\\
	R_{\gso}(\hat\pi_{n-1})R_{\gso}(\hat\pi_{n-2}) =&
	R_{\gso}(\hat\pi_{n-1}+\hat\pi_{n-2})
	+R_{\gso}(\hat\pi_{n-1})+R_{\gso}(\hat\pi_{n-3}).
\end{align*}
Therefore
\[R_{\gsl}(\pi_{n-1}+\pi_{n-2})=R_{\gso}(\hat\pi_{n-1}+\hat\pi_{n-2})
+R_{\gso}(\hat\pi_{n-1}).\]
Next we have
\begin{align*}
	R_\gsl(\pi_{n-1}+\pi_{n-2}+\pi_2) \subset&
	R_{\gsl}(\pi_{n-1}+\pi_{n-2})R_{\gsl}(\pi_2)\\
	=&
	(R_{\gso}(\hat\pi_{n-1}+\hat\pi_{n-2})+
	R_{\gso}(\hat\pi_{n-1}))R_{\gso}(\hat\pi_2)
	\\
	=&
	(R_{\gso}(\hat\pi_{1}+\hat\pi_{2})+
	R_{\gso}(\hat\pi_{1}))R_{\gso}(\hat\pi_2).
\end{align*}
Finally,
\begin{align*}
	R_{\gso}(&\hat\pi_{1})R_{\gso}(\hat\pi_{2}) =
	R_{\gso}(\hat\pi_{1}+\hat\pi_{2})
	+R_{\gso}(\hat\pi_{1})+R_{\gso}(\hat\pi_{3}).
	\\
R_{\gso}(&\hat\pi_{1}+\hat\pi_{2}) 	R_{\gso}(\hat\pi_{2})
	=	R_{\gso}(\hat\pi_{3})+R_{\gso}(\hat\pi_{2}+\hat\pi_{3})
	+	R_{\gso}(\hat\pi_{1}) + 	R_{\gso}(\hat\pi_{1}+\hat\pi_{4})
	\\
	&+ 2	R_{\gso}(\hat\pi_{1}+\hat\pi_{2})
	+ R_{\gso}(2\hat\pi_{1}+\hat\pi_{3})
	+ R_{\gso}(2\hat\pi_{1}+2\hat\pi_{2})
	+	R_{\gso}(3\hat\pi_{1}).
\end{align*}
The last expression shows that the complexification of  $H^2(\g_-,\bar\g)$ doesn't contain
any trivial submodules.

\textbf{2.}  As $\gso(n,\C)$-module the homogeneity 1 component in $H^1(\g_-,\bar{\g}/\g)\ot\C$ is a part of:
\[
S^2(\pi_1)R(\pi_1)=(R(2\pi_1)+R(\pi_0))R(\pi_1)
=R(3\pi_1)+R(\pi_1+\hat{\pi}_2)+ 2R(\pi_1).
\]
As we see this component doesn't contain any trivial submodule.

\textbf{3.} The homogeneity 2 component in $H^1(\g_-,\bar{\g}/\g)$ is a
symmetric 2-tensor and as we know:
\[ S^2_\gso(\pi_1)=R_\gso(2\pi_1)+R_\gso(\pi_0).\]
Therefore the family of constant curvature structures is not more than $1$-dimensional in the generic case.

The corresponding cohomology element invariant under the action of $\gso(n,\R)$ is
\[ \alpha=\sum_{i=1}^n \alpha_{(ii})=\sum_{i,t=1}^n \left( e_t\ot
e_{[ti]}^*\we e_i^*+ \sum_{r=1}^n e_{[tr]}\ot e_{[ti]}^*\we e_{[ir]}*
\right) .\]
One could check that
constant curvature models with curvature $k\alpha$ are
isomorphic to $SO(n+1)$ if $k>0$ and to $SO(1,n)$ if $k<0$.

\textbf{4.} In the case $n=4$ the computations are slightly different
since
 $\gso_4(\C)=\gsl_2(\C)\times\gsl_2(\C)$. The fundamental weights
 are $(i\pi_1,j\pi_1)$. We use the notation
 \begin{align*}
 & R(\hat{\pi}_1) =R( \pi_1,\pi_1),
 \\
 & R(\hat{\pi}_2) = \Lambda^2(\hat{\pi}_1) = R(\pi_0,2\pi_1)+
 R(2\pi_1,\pi_0),
 \\
 & R(\hat{\pi}_3) =\Lambda^3(\hat{\pi}_1) = R( \pi_1,\pi_1 ).
 \end{align*}
 As in previous computations $H^2$-part of the harmonic curvature belongs to
 \begin{align*}	&\left(R(\hat\pi_{1}+\hat\pi_{2})+
 R(\hat\pi_{1})\right)R(\hat\pi_2)
 \\
 &= ( R (\pi_1,3\pi_1) + R (3\pi_1,\pi_1)
 + R (\pi_1,\pi_1) ) (R(\pi_0,2\pi_1)+
 R(2\pi_1,\pi_0))
 \\
 &=R (\pi_1,5\pi_1) + R (5\pi_1,\pi_1) + 2R (3\pi_1,3\pi_1)
 \\
 &+ 3 R
 (\pi_1,3\pi_1) + 3R (3\pi_1,\pi_1) + 4 R (\pi_1,\pi_1)
 \end{align*}
 and doesn't have trivial components.  Homogeneity 1 component in
 $H^1(\g_-,\bar{\g}/\g)$ belongs to
 \[
 S^2(\hat\pi_1)R(\hat\pi_1)=R (3\pi_1,3\pi_1) + R (\pi_1,3\pi_1) + R
 (3\pi_1,\pi_1) + 2R (\pi_1,\pi_1)
 \]
 and doesn't have any trivial components as well. Homogeneity 2 component in
 $H^1(\g_-,\bar{\g}/\g)$ has 1-dimensional trivial submodule
 \[ S^2(\hat\pi_1)=R(2\hat\pi_1)+R(\hat\pi_0).\] Explicit check shows
 that again we obtain unique up to scaling constant curvature models
 on $SO(5)$ and $SO(4,1)$.

\textbf{5.} Consider the case $n=3$. In this case the complexification of $\g_0=\gso(3,\R)$ is $\gsl(2,\C)$. As $\gsl(2,\C)$-module the complexification of $\g_{-1}$ and $\g_{-2}$ are $R(2\pi_1)$.

First of all, $H^2$-part of the harmonic curvature is $1$-dimensional and therefore is trivial under the action of $\gso(3)$. Secondly, homogeneity 1 component in  $H^1(\g_-,\bar{\g}/\g)\ot\C$ belongs to
 \[
 S^2(2 \pi_1)R(2\pi_1)=R (6\pi_1) + R (4\pi_1) + 2 R
 (2\pi_1)
 \]
 and doesn't have any trivial components. Homogeneity 2 component in
 $H^1(\g_-,\bar{\g}/\g)\ot\C$ has 1-dimensional trivial submodule
 \[ S^2(2 \pi_1)=R(4\pi_1)+R(\pi_0).\]

To sum up the family of constant curvature structures is $2$-dimensional.

\end{proof}

\begin{rem}
One can use Theorems \ref{T:freeH2} and \ref{T:freeH1} in order to write down homogeneous models explicitly for free sub-Riemannian structures of rank 3. The corresponding family is a cone over a circle with the vertex of the cone being a flat model. All models apart 2 lines passing through the vertex of the cone have 9-dimensional semisimple symmetry Lie algebras. The symmetry algebras of models on the remaining 2 lines have a 6-dimensional semisimple part in the Levi decomposition.
\end{rem}

\section{Appendix --- Admissible
 Cartan connections for filtered $G_0$-structures with trivial prolongation}

The goal of this section is to provide a simple and straightforward
construction of canonical Cartan connections under algebraic assumptions
relevant for the situations dealt with in the paper. In particular we show
that such construction is almost identical to the well known theory of
classical G-structures.

Exactly as in Section 2 above, let $D$  be  a bracket  generating  $n$-dimensional distribution  on  an
$m$-dimensional manifold $M$ and $D_{-1} =D \subset D_{-2} \subset \cdots
  \subset D_{-k} = TM$ the associated filtration.  As always, we assume that
  the distribution $D$ has constant symbol, that is for any $x \in M$ the
  negatively graded Lie algebra
$$\operatorname{gr}T_xM = (T_xM)_{-k} + \cdots + (T_xM)_{-1}=
(D_{-k}/ D_{-k+1})_x + \cdots + (D_{-1})_x$$
is isomorphic  to  some  fixed symbol  algebra
$\g_{-} = \g_{-k} + \cdots+ \g_{-1}$.

Further, we fix  a  subalgebra $\g_0 \subset \mathfrak{der}{\g_-}$ of
(graded) derivations, which integrates to a Lie group $G_0 \subset
\operatorname{Aut}(\g_-)$ of automorphisms and we denote by
$$\pi : \mathcal{F} \to M $$
the associated filtered $G_0$-structure  (or  graded $G_0$   structure), i.e.  $G_0$-principal bundle
of  graded  frames $f : \g_- \to \operatorname{gr}T_xM$ (isomorphisms of the
graded Lie algebras) in the graded tangent space $\operatorname{gr}TM $.
Let us assume that the Tanaka prolongation $\g^{(1)}$ of the non positively
graded Lie algebra $\g:=\g_{\leq 0} = \g_- + \g_0$ is trivial.  As
detailed in Section 2, it is  the  case  if  $\mathfrak g_0\subset \mathfrak {so}(\mathfrak
g_{-1})$, and   then there is a unique normal Cartan connection for this filtered $G_0$-structure.

%
%
To  simplify   the  exposition,  we  will  assume  that  the
depth $k =2$, i.e.  the  symbol algebra
$\g_- = \g_{-2} + \g_{-1}$ is  metaabelian.
Then  a complementary  distribution $D'$ to  $D$, such  that
$TM  = D \oplus D'$,  allows   to  identify  graded frames   with  tangent  frames    and  identify
 the filtered $G_0$-structure  with a  standard  $G_0$-structure.

Recall   that  a $G$-structure on  a principal $G$-bundle   $\pi :
\mathcal{F} \to M$ over $m$-dimensional manifold $M$ is defined by a
soldering 1-form $\theta : T \mathcal{F} \to \mathbb{R}^m$, that is strictly
horizontal ($\mathrm{ker}\theta = T^{vert}\mathcal{F}$) $G$-equivariant
vector-valued 1-form.
%
%

Note  that for  the filtered $G_0$-structure  there  is  the natural  projection
  $$ T_xM \to (T_xM)_{-2}= T_xM /D_x, \,\,  X \mapsto X_{-2} $$
  which  defines  the   partial $\g_{-2}$-valued  soldering  form $\theta_{-2} : T\mathcal{F} \to \g_{-2}$ on the filtered $G_0$-structure $\mathcal{F}$  defined  by
    $\theta_{-2}(X_f) =  f^{-1} (\pi_*X)_{-2}$.

A $G_0$-structure  in  $\pi : \mathcal{F} \to  M$   defined  by   an
  extension of the partial soldering form $\theta_{-2}$ to a $\g_- = \g_{-2}
  + \g_{-1}$-valued soldering form is called \emph{admissible}.
 A complementary distribution $D'$  defines an isomorphism
$$\mathrm{gr}_{D'}: TM \to \operatorname{gr}TM ,
   X \mapsto \mathrm{gr}(X) = X_{-2} \oplus X_{-1}                                  $$
where $X = X_{-2} + X_{-1} \in T_xM = D' \oplus D$ is the natural decomposition
  of $X $ and we identify $D'$ with $(TM)_{-2} = TM/D$  and   extends graded  frames  to an   admissible   frames.
  We  get

\bl  Any   complementary  distribution $D'$ to $D$  (s.t.
$TM = D \oplus  D'$) defines  an  admissible  $G_0$-structure on  the principal bundle  $\pi : \mathcal{F} \to M$ and   any admissible  $G_0$-structure is  associated  with  some $D'$.
The  soldering  form $\theta = \theta_{D'}$  of   such $G_0$-structure  is  given  by
$$   \theta_f (X) = f^{-1}\mathrm{gr}(\pi_* X) \in \g_-.  $$
     \el

Let us observe that there always is the distinguished choice of $D'$ given
by the co-closeness curvature normalization for the sub-Riemannian
structures, as explained in detail in Section 2.
In particular examples, the identification of $D'$
usually amounts to minimizing traces of individual components of the
curvature. We have illustrated this in the detailed discussion of the
2-step free distributions in the end of Section 5.

\subsection{Construction  of a Cartan  connection}

%
Let $\pi : \mathcal{F} \to M$ be a filtered $G_0$-structure.
We  consider   the  $G_0$-structure  on  $M$  defined   by a complementary  distribution  $D'$.
Each  element $f \in \mathcal{F}$ is identified  with  an  admissible  frame $f : \g_- \to
\operatorname{gr}T_xM = T_xM$.
 We  denote  by $f^{-1}: T_xM \to \g_-$
the  dual  coframe.

\medskip
\noindent{\it Algebraic assumptions.}
\begin{enumerate}
\item[i)]
We assume  that the  first prolongation  $\g_0^{(1)} =0$ of  the linear Lie
algebra $\g_0 \subset \mathfrak{gl}(\g_-)$ is trivial.  It is the case if
the first prolongation of $\g= \g_- + \g_0$ is trivial.
\item[ii)] We  assume  that   there  is  a fixed $G_0$-invariant
complement  $\mathcal{T}$  to  the  subspace   $\delta (\g_0 \otimes \g_-^*)$
in the vector  space of  torsions  $\operatorname{Tor} :=
\g_- \otimes \Lambda^2 \g_-^*$ (here $\delta$ means the usual Spencer
differential).
\end{enumerate}

According  to    theory of  $G$-structures, under these assumptions
there is  a unique
connection $\omega : T\mathcal{F} \to \g_0$ on the principal bundle $ \pi :
\mathcal{F} \to M$ whose torsion function $ t : \mathcal{F} \to Tor $ take
values in $\mathcal{D}$.

If $H = \mathrm{ker} \omega_f$ is  the  horizontal  space of  the  connection    and $\pi_H : =(\pi_*)_H :H \to T_xM$ is  the  natural isomorphism
then   the  value $t_f \in Tor$  is  given  by
$$   t_f (u,v) = d \theta_f (\pi_H^{-1} fu,\pi_H^{-1} fv ) \in \g_- $$
for  any  $u,v \in  \g_-$.
In  this  case,  1-form
$$ \varpi = \omega + \theta : T \mathcal{F} \to \g=
\g_0 + \g_- $$
is a Cartan connection modelled on the affine space $ \g_-
=(G_0 \rtimes \g_-)/G_0$ that is a $G_0$-equivariant map without kernel
(such that $ \varpi_f : T_f\mathcal{F} \to \g$ is a vector space
isomorphism) which extends the vertical parallelism $T^{vert}
\mathcal{F} \to \g_0$.

Note   that     we  may  consider the  same  1-form   $ \varpi =   \omega +
\theta $ as a Cartan connection in $\mathcal{F}$ modelled on the homogeneous
manifold $G_- = G/G_0 =(G_0\rtimes G_-)/G_0$ where $ G= G_0\rtimes G_-$ is a
semidirect product of the nilpotent Lie group, generated by the Lie algebra
$\g_-$ and a group $G_0 \subset \mathrm{Aut}(G_-)$ of its automorphisms.
Such Cartan connections $\varpi$ on the filtered $G_0$-structure $\mathcal{F}
\to M$ are called \emph{admissible}.

We may summarize:

 \bp Any  admissible Cartan  connection  is  obtained  in this  way  and it
 is associated with a complementary distribution $D'$.  \ep

 The  curvature
  $$\kappa = d \varpi + \frac12 [\varpi, \varpi]$$
   of  this  connection  is a 2-form on
$\mathcal{F}$ valued in $\g = (\g_- + \g_0)$, which describes   the  deviation  of  the bundle   $\pi : \mathcal{F} \to M$ from  the model  bundle $G \to  G/G_0$.
 Here  $[\phi, \psi]$ denote  the bracket  of $\g$-value  1-forms,  defined as  2-form  by
 $$  [\phi, \psi](X,Y) = [\phi(X), \psi(Y)] - [\psi(Y), \phi(X)] . $$

  The  curvature  form
   is  decomposed into $\g_0$-valued  form $\Omega$  called  the  curvature
   and the $\g_-$-valued part $\Theta$, called torsion.  The general
Cartan structure equations split nicely, too:
$$  \Omega = d \omega  + \frac12 [\omega,  \omega]   $$
 $$    \Theta  = d \theta + \frac12 [\theta, \theta] + [\omega, \theta].$$
 The last  equation   may  be  decomposed  into $\g_{-1}$   and  $\g_{-2}$ -part  as  follows:
 $$ \Theta_{-1} = d \theta_{-1} + [\omega,  \theta_{-1}]                      $$
 $$ \Theta_{-2} = d \theta_{-2} + \frac12 [\theta_{-1}, \theta_{-1}] + [\omega,  \theta_{-2}]                      $$
  Note  that  the  2-forms  $\kappa, \Omega, \Theta$  can   be  considered    as
the $G_0$-invariant  functions
  $$K: \mathcal{F} \to \Lambda^2(\g)^* \otimes \g, \,R= \mathcal{F}\to \Lambda^2(\g)^* \otimes \g_0, \, T:\mathcal{F} \to \Lambda^2(\g)^* \otimes \g_-  $$
  called  Cartan  curvature  function, curvature  function  and  torsion   function respectively.

\end{document}